\documentclass[12pt, reqno]{amsart}
\usepackage{amsfonts}
\usepackage{amssymb}
\usepackage{amsfonts,mathrsfs,amsmath,amssymb}

\theoremstyle{plain}
\newtheorem{theo}[equation]{Theorem}
\newtheorem{pro}[equation]{Proposition}
\newtheorem{coro}[equation]{Corollary}
\newtheorem{lem}[equation]{Lemma}
\newtheorem{defi}[equation]{Definition}
\newtheorem{rem}[equation]{Remark}
\def\vn{\varepsilon}
\def\ot{\otimes}
\def\om{\omega}
\def\mg{\mathfrak{g}}

\def\lan{\langle}
\def\ran{\rangle}
\def\al{\alpha}

\def\be{\beta}
\def\De{\Delta}
\def\ga{\gamma}

\def\si{\sigma}
\def\om{\omega}
\def\vep{\varepsilon}

\def\de{\delta}
\def\pa{\partial}
\def\La{\Lambda}
\def\la{\lambda}

\def\O{{\mathcal O}}

\def\bN{{\mathbb N}}
\def\bZ{{\mathbb Z}}
\def\bQ{{\mathbb Q}}

\def\bK{{\mathbb K}}
\def\bC{{\mathbb C}}
\def\lra{\longrightarrow}

\def\dim{\mbox{\rm dim}\,}

\def\pa{\partial}
\def\lra{\longrightarrow}
\def\al{\alpha}
\def\be{\beta}
\def\ga{\gamma}
\def\si{\sigma}

\def\De{\Delta}
\def\de{\delta}

\def\vn{\varepsilon}

\def\ot{\otimes}
\def\la{\lambda}
\def\La{\Lambda}
\def\om{\omega}
\def\lra{\longrightarrow}
\def\mg{\mathfrak {g}}

\def\sm1{_{(-1)}}
\def\s1{_{(1)}}
\def\s2{_{(2)}}
\def\s3{_{(3)}}
\def\s0{_{(0)}}
\def\um1{^{(-1)}}
\def\u1{^{(1)}}
\def\u2{^{(2)}}
\def\u3{^{(3)}}
\def\u0{^{(0)}}

\def\wt{\mbox{\rm wt}}
\def\e{\mbox{\rm e}}
\begin{document}

\title[two-parameter quantum groups, (II)]
{Notes on two-parameter quantum groups, (II) }

\author[Hu]{Naihong Hu$^\star$}
\address{Department of Mathematics, East China Normal University,
Shanghai 200062, PR China} \email{nhhu@math.ecnu.edu.cn}
\thanks{$^\star$N. Hu,
supported in part by the NNSFC (Grants 10431040, 10728102), the
PCSIRT, the National/Shanghai Leading Academic Discipline Project
(Project Number: B407).}

\author[Pei]{Yufeng Pei$^\dag$}
\address{Department of Mathematics, Shanghai Normal University, Guilin Road 100,
Shanghai 200234 PR China}\email{peiyufeng@gmail.com, pei@shnu.edu.cn}
\thanks{$^\dag$Y. Pei,
 supported in part by the NNSFC (Grant
10571119), the ZJNSF (Grant Y607136) and the Leading Academic Discipline Project of Shanghai
Normal University (Grant DZL803).}

\begin{abstract}
This paper is the sequel to \cite{HP1} to study the deformed
structures and representations of two-parameter quantum groups
$U_{r,s}(\mathfrak{g})$ associated to the finite dimensional simple
Lie algebras $\mg$.
An equivalence of the braided tensor categories $\O^{r,s}$
and $\O^{q}$ is explicitly established.

\end{abstract}
\maketitle

\section{Introduction}

In \cite{HP1}, the authors introduced an unified
definition for a class of two-parameter quantum groups
$U_{r,s}(\mg)$ associated to finite-dimensional simple Lie algebras
$\mg$ in terms of the Euler form and showed that the positive parts
of quantum groups are $2$-cocycle deformations of each other as
graded associative algebras if two parameters $r, s$ satisfy
certain conditions. This work is a continuation of the paper
\cite{HP1} to characterize the structure of  two-parameter quantum groups
$U_{r,s}(\mg)$ and  the category of
$U_{r,s}(\mg)$-modules.

In this paper, with the help of  $(r,s)$-skew derivations introduced
in \cite{HP1}, we prove the positive part $U_{r,s}^+$ has a natural
$U_{r,s}(\mg)$-module algebra structure. Partially motivated by
Doi-Takeuchi's \cite{DT} and Majid's \cite{Ma} results on
 Hopf $2$-cocycle deformation theory, we show that
the two-parameter quantum groups $U_{r,s}(\mg)$ can be obtained from the one-parameter quantum
group $U_{q,q^{-1}}(\mg)$ by twisting the multiplication via an
explicit Hopf $2$-cocycle $\si$, that is,
\begin{equation}
U_{r,s}(\mg)\simeq U_{q,q^{-1}}^{\si}(\mg),\quad (\text{as Hopf
algebras}).
\end{equation}
It is noticed that this kind of deformation of the algebra structure
depends on its coalgebra structure. Inspired by
Hodges-Levasseur-Toro's \cite{HLT} work on the multi-parameter
quantum groups, we prove that $U_{r,s}(\mg)$ can be deformed from
$U_{q,q^{-1}}(\mg)$ as bigraded structures by twisting the
multiplication via a bicharacter $\zeta $ of the free abelian group
$Q\times Q$ (where $Q$ is the root lattice of $\mg$):
\begin{equation}
U_{r,s}(\mg)\simeq U_{q,q^{-1},\zeta}(\mg),\quad (\text{as
$Q$-bigrading Hopf algebras}),
\end{equation}
which recovers Theorem 3.3 in \cite{HP1} when it is restricted to
the positive part of $U_{r,s}(\mg)$. As an application, we give a new
and simple proof for the existence of nondegenerate skew Hopf
pairing on $U_{r,s}(\mg)$, which were studied previously in
\cite{BW1,BGH1} where verifying the $(r,s)$-Serre relations to be
preserved resulted in rather involved formulas.

Representation theory of two-parameter quantum groups $U_{r,s}(\mg)$
under the assumption $rs^{-1}$ being nonroot of unity has been
investigated in \cite{BW2, BGH2}, for $\mg$ classical. It was showed
that the category $\O^{r,s}$ of finite-dimensional weight
$U_{r,s}(\mg)$-modules (of type $1$) is a semisimple braided tensor
category. A natural question is to find the explicit relations
between the categories $\O^{r,s}$ and $\O^{q}$, where $\O^{q}$ is
the category of finite-dimensional weight
$U_{q,q^{-1}}(\mg)$-modules (of type $1$) (\cite{HZ}) and $\mg$ is
of finite type. Our main theorem is

\smallskip

{\bf Theorem} {\it As braided tensor categories, the categories
$\O^{r,s}$ and $\O^{q}$ are equivalent.
}

\smallskip


This paper is organized as follows. In Section 2, we recall the
definition of the two-parameter quantum groups given in \cite{HP1}
and some basic properties. In Section 3, we show that $U_{r,s}^+$ is
a $U_{r,s}(\mg)$-module algebra. Section 4 is devoted to the study
of a certain Hopf $2$-cocycle deformation of $U_q(\mathfrak g)$. In
Section 5 we discuss bigraded deformation of Hopf algebras.  In
Section 6, we give a new and simple proof for the existence of
nondegenerate skew Hopf pairing on  $U_{r,s}(\mg)$ and obtain an
equivalence of the braided tensor categories.


\section{\bf Two-parameter quantum groups}

\subsection{}
Let us start with some notations. For $n>0$, let
\begin{gather*}
(n)_v=1+v+\cdots+v^{n-1}=\frac{v^n-1}{v-1}.\\
(n)_v!=(1)_v(2)_v\cdots(n)_v
\quad \text{and}\quad (0)_v!=1.
\\
\binom{n}{k}_v=\frac{(n)_v!}{(k)_v!(n-k)_v!},\quad
[n]_v=\frac{v^n-v^{-n}}{v-v^{-1}}.
\\
[n]_v!=[1]_v[2]_v\cdots[n]_v ,\quad [0]_v!=1, \quad  \left[n\atop
k\right]_v=\frac{[n]_v!}{[k]_v![n-k]_v!}.
\end{gather*}
Throughout the paper, we denote by $\bZ,\, \bZ_+,\, \bN$, $\bC$ and
$\bQ$ the sets of integers, of positive integers, of non-negative
integers, of complex numbers and of rational numbers, respectively.

\subsection{}
Let $\mg$ be a finite-dimensional simple Lie algebra over a field
$\bK\supseteq\bQ$ and $A=(a_{ij})_{i,j\in I}$ be an associated
Cartan matrix. Let $d_i$ be relatively prime positive integers such
that $d_ia_{ij}=d_{j}a_{ji}$ for $i,j\in I$. Let $\Pi=\{\al_i\mid
i\in I\}$ be the set of simple roots, $Q=\bigoplus_{i\in I}\bZ\al_i$
root lattice, $Q^+=\bigoplus_{i\in I}\bN\al_i$ positive root
lattice, $\La$ weight lattice, and $\La^+$ the set of dominant
weights. Let $\Phi$ be the set of roots and $\Phi^+$ positive roots.
Let $\bQ(r,s)$ be the rational functions field in two variables
$r,\,s$ over $\bQ$. Let $r_i=r^{d_i}, \,s_i=s^{d_i}$ for $i\in I$.
Now let $\bK\supseteq\bQ(r,s)$ be a field such that
$(rs^{-1})^{\frac{1}{m}}\in\bK$ for some $m\in\bZ_{+}$ such that
$m\La\subseteq Q$ for the possibly smallest positive integer $m$. We
always assume that $rs^{-1}$ is not a root of unity. Let
$\lan-,-\ran$ be the Euler bilinear form on $Q\times Q$ defined by
\begin{eqnarray}
\lan i,j\ran:=\lan\al_i,\al_j\ran=
\begin{cases}
&d_ia_{ij}\quad i<j,\\
&d_i\quad\quad\, i=j,\\
&0\quad\quad\ \   i>j.
\end{cases}
\end{eqnarray}
For $\lambda\in\La$, we linearly extend the bilinear form
$\lan-,-\ran$ to be defined on $\La\times \La$ such that
$\lan\lambda,i\ran=\frac1{m}\sum_{j=1}^na_j\lan j,i\ran$, or $\lan
i,\lambda\ran=\frac1{m}\sum_{j=1}^na_j\lan i,j\ran$ for
$\lambda=\frac1{m}\sum_ja_j\al_j$ with $a_j\in\mathbb Z$.
\begin{defi}[Hu-Pei \cite{HP1}]
The two-parameter quantum group $U_{r,s}(\mg)$ is a unital
associative algebra over $\bK$ generated by
$e_i,f_i,\om_i^{\pm1},\om_i'^{\pm1},$ $i\in I$, subject to the
relations:
\begin{eqnarray*}
&(R1)&\qquad \om_i^{\pm1}\om_i^{\mp1}=\om_i'^{\pm1}\om_i'^{\mp1}=1,\\
&(R2)&\qquad [\om_i,\om_j]=[\om_i',\om_j']=[\om_i,\om_j']=0,\\
&(R3)&\qquad \om_ie_j=r^{\lan j, i\ran}s^{-\lan i, j\ran} e_j\om_i,
\qquad\qquad\om_i'e_j=r^{-\lan i,j\ran}s^{\lan j,i\ran} e_j\om_i'.
\\
&(R4)&\qquad\om_if_j=r^{-\lan j, i\ran}s^{\lan i, j\ran}
f_j\om_i,\qquad\qquad \om_i'f_j=r^{\lan i,j\ran}s^{-\lan
j,i\ran}f_j\om_i'.\\
&(R5)&\qquad
e_if_j-f_je_i=\delta_{i,j}\frac{\om_i-\om_i'}{r_i-s_i}.
\\
&(R6)&\qquad\sum_{k=0}^{1-a_{ij}} (-1)^{k}
\binom{1-a_{ij}}{k}_{r_is_i^{-1}} c^{(k)}_{ij}
e_{i}^{1-a_{ij}-k}\,e_{j}\, e_{i}^{k} =0, \qquad\,\; (i\neq
j),\end{eqnarray*}
\begin{eqnarray*}
&(R7)&\qquad \sum_{k=0}^{1-a_{ij}} (-1)^{k}
\binom{1-a_{ij}}{k}_{r_is_i^{-1}} c^{(k)}_{ij}\, f_{i}^{k}\,f_{j}\,
f_{i}^{1-a_{ij}-k} =0, \qquad (i\neq j),
\end{eqnarray*}
where
$$
c^{(k)}_{ij}=(r_is_i^{-1})^{\frac{k(k-1)}{2}} r^{k\lan
j,i\ran}s^{-k\lan i,j\ran},\qquad\textit{for } \ i\neq j.
$$
\end{defi}

The algebra $U_{r,s}(\mg)$ has a Hopf algebra structure
with the comultiplication, the counit and the antipode given by:
\begin{gather*}
\Delta(\om_i^{\pm1})=\om_i^{\pm1}\ot\om_i^{\pm1}, \qquad
\Delta({\om_i'}^{\pm1})={\om_i'}^{\pm1}\ot{\om_i'}^{\pm1},\\
\Delta(e_i)=e_i\ot 1+\om_i\ot e_i, \qquad \Delta(f_i)=1\ot
f_i+f_i\ot \om_i',\\
\vn(\om_i^{\pm1})=\vn({\om_i'}^{\pm1})=1, \qquad\qquad
\vn(e_i)=\vn(f_i)=0,\\
S(\om_i^{\pm1})=\om_i^{\mp1}, \qquad\quad\qquad
S({\om_i'}^{\pm1})={\om_i'}^{\mp1},\\
S(e_i)=-\om_i^{-1}e_i,\qquad S(f_i)=-f_i\,{\om_i'}^{-1}.
\end{gather*}

Let $U_{r,s}^+$ (respectively, $U_{r,s}^-$) be the subalgebra of
$U_{r,s}:=U_{r,s}(\frak g)$ generated by the elements $e_i$
(respectively, $f_i$) for $i\in I$, and $U^0_{r,s}$ the subalgebra of
$U_{r,s}$ generated by $\om_i^{\pm1},\om_i'^{\pm1}$  for $i\in I$.
Moreover, let $U_{r,s}^{\geq0}$ (respectively, $U_{r,s}^{\leq0}$)
be the subalgebra of $U_{r,s}$ generated by the elements
$e_i,\,\om_i^{\pm1}$ for $i\in I$ (respectively, $f_i,\,
\om_i'^{\pm1}$ for $i\in I$). For each $\mu\in Q$ (the root
lattice of $\frak g$), we define elements $\om_{\mu}$ and
$\om_{\mu}'$ by
$$
\om_{\mu}=\prod_{i\in I}\om_{i}^{\mu_i},\quad \om_{\mu}'=\prod_{i\in
I}\om_{i}'^{\mu_i},\qquad\textit{for } \ \mu=\sum_{i\in
I}\mu_i\al_i\in Q.
$$
For $\be\in Q^+$, let
$$
(U_{r,s}^{\pm})_{\pm\be}=\left\{x\in
U_{r,s}^{\pm}\,\left|\,\om_{\mu}x\om_{-\mu}
=r^{\lan\be,\mu\ran}s^{-\lan\mu,\be\ran}x,\, \om_{\mu}'x\om_{-\mu}'
=r^{-\lan\mu,\be\ran}s^{\lan\be,\mu\ran}x,\, \forall\ \mu\in
Q\right\}\right.,
$$
then
$$U_{r,s}^{\pm}=\bigoplus_{\be\in Q^+}(U_{r,s}^{\pm})_{\pm\be}$$ are
$Q^+$-graded.

\subsection{}
Assume that $r=s^{-1}=q$, it is clear that $U_{r,s}(\mg)$ becomes
the one-parameter quantum group $U_{q,q^{-1}}(\mg)$ of
Drinfel'd-Jimbo type with double group-like elements, which is a
unital associative algebra over $\bK$ generated by
$E_i,F_i,K_i^{\pm1},K_i'^{\pm1},(i\in I)$, subject to the following
relations:
\begin{eqnarray*}
 &(1)&\quad
K_i^{\pm1}K_j^{\pm1}=K_j^{\pm1}K_i^{\pm1},\qquad
K_i^{\pm1}K_i^{\mp1}=K_i'^{\pm1}K_i'^{\mp1}=1,\\
&(2)&\quad K_iE_jK_i^{-1}=q_i^{a_{ij}}
E_j, \qquad K_i'E_jK_i'^{-1}=q_i^{-a_{ij}}
E_j,
\\
&(3)&\quad K_iF_{j}K_i^{-1}=q_i^{-a_{ij}}
F_{j},\qquad
K_i'F_{j}K_i'^{-1}=q_i^{a_{ij}}F_{j},\\
&(4)&\quad[\,E_i, F_{j}\,]=\delta_{i,j}\frac{K_i-K_i'}{q_i-q_i^{-1}},\\
&(5)&\quad\sum_{k=0}^{1-a_{ij}} (-1)^{k} {\left[{1-a_{ij}}\atop{k}\right]}_{q_i}
 E_{i}^{1-a_{ij}-k} E_{j}
E_{i}^{k} =0, \quad i\neq j,\\
&(6)&\quad\sum_{k=0}^{1-a_{ij}} (-1)^{k} {\left[{1-a_{ij}}\atop{k}\right]}_{q_i} F_{i}^{1-a_{ij}-k} F_{j}
F_{i}^{k} =0, \quad i\neq
j.
\end{eqnarray*}
Moreover, it is known that $U_{q,q^{-1}}(\mg)$ has a Hopf algebra
structure with the comultiplication, the counit and the antipode
given by:
\begin{align}
& \Delta(K_i^{\pm1})=K_i^{\pm1}\ot K_i^{\pm1}, \qquad\ \
\Delta({K_i'}^{\pm1})={K_i'}^{\pm1}\ot{K_i'}^{\pm1},\nonumber\\
&\Delta(E_i)=E_i\ot 1+K_i\ot E_i, \qquad \Delta(F_i)=1\ot
F_i+F_i\ot K_i',\nonumber\\
&\vn(K_i^{\pm1})=\vn({K_i'}^{\pm1})=1, \qquad\quad\
\vn(E_i)=\vn(F_i)=0,\\
&S(K_i^{\pm1})=K_i^{\mp1}, \qquad\qquad\qquad\
S({K_i'}^{\pm1})={K_i'}^{\mp1},\nonumber\\
&S(E_i)=-K_i^{-1}E_i,\qquad\qquad\qquad\  S(F_i)=-F_i\,{K_i'}^{-1}.\nonumber
\end{align}

By abuse of notation, we denote $U_{q}(\mg):=U_{q,q^{-1}}(\mg)$.
Let $U_{q}^+$ (respectively, $U_{q}^-$) be the subalgebra of
$U_{q}:=U_{q}(\frak g)$ generated by the elements $E_i$
(respectively, $F_i$) for $i\in I$, and $U^0_{q}$ the subalgebra of
$U_{q}$ generated by $K_i^{\pm1},K_i'^{\pm1}$  for $i\in I$.
Moreover, let $U_{q}^{\geq0}$ (respectively, $U_{q}^{\leq0}$) be the
subalgebra of $U_{q}$ generated by the elements $E_i,\,K_i^{\pm1}$
for $i\in I$ (respectively, $F_i,\, K_i'^{\pm1}$ for $i\in I$). For
each $\mu\in Q$ (the root lattice of $\frak g$), we denote $K_{\mu}$
and $K_{\mu}'$ by
$$
K_{\mu}=\prod_{i\in I}K_{i}^{\mu_i},\quad K_{\mu}'=\prod_{i\in
I}K_{i}'^{\mu_i},\qquad\text{for } \ \mu=\sum_{i\in I}\mu_i\al_i\in
Q.
$$
For $\be\in Q^+$, let
$$
(U_{q}^{\pm})_{\pm\be}=\left\{x\in U_{q}^{\pm}\,\left|\, K_{\mu}x
K_{-\mu} =q^{(\mu,\be)}x,\, K_{\mu}'x K_{-\mu}' =q^{-(\mu,\be)}x,\,
\forall\ \mu\in Q\right\}\right.,
$$
then $U_{q}=\bigoplus_{\be\in Q}(U_{q})_{\be}$ is a $Q$-graded
algebra. Moreover, there exists a unique nondegenerate skew Hopf
pairing $\lan\,,\,\ran_{q}$ between $U_{q}^{\leq0}$ and
$U_{q}^{\geq0}$, such that
\begin{align}
&\lan F_i, E_j\ran_q=\de_{i,j}\frac{1}{q_i^{-1}-q_{i}},\\
&\lan xK_{\mu}',yK_{\nu}\ran_q=q^{(\mu,\nu)}\lan x,y\ran_{q}
\end{align}
for any $i, j\in I, x\in U_{q}^{\leq0}, y\in U_{q}^{\geq0},
\mu,\nu\in Q$. Therefore, $U_{q}$ has the Drinfel'd double
structure, that is, as Hopf algebras, we have the following
isomorphism:
\begin{eqnarray}
U_{q}\simeq\mathcal{D}(U_{q}^{\leq0},U_{q}^{\geq0},\lan\,,\,\ran_{q}).
\end{eqnarray}

\section{\bf $U_{r,s}(\mg)$-module algebra structure over $U_{r,s}^+$}
Let $(H,m,1,\De,\vep,S)$ be a Hopf algebra over a field $k$. Recall
that an (associative) algebra $A$ over $k$ an $H$-module algebra if
$A$ has an (left) $H$-module structure such that
\begin{eqnarray*}
(1)\quad h1_A=\vep(h)1_A, \qquad (2)\quad  h(ab)=\sum
(h_{(1)}a)(h_{(2)}b),
\end{eqnarray*}
for $h\in H$, $a, b \in A$ with $\De(h)=\sum h_{(1)}\otimes
h_{(2)}$. Here the second condition means that the multiplication is
a homomorphism of $H$-modules.

\subsection{Skew derivations}
By the definition of coproduct, we have
\begin{equation*}
\Delta(x)\in\bigoplus_{0\le\nu\le\be}(U_{r,s}^+)_{\be-\nu}\om_\nu\otimes (U_{r,s}^+)_{\nu},\quad\forall\  x\in
(U_{r,s}^+)_{\be}.
\end{equation*}
For $i\in I$ and $\be\in Q^+$, define linear maps
$${\hat\pa_i},\,{_i\hat\pa}: (U_{r,s}^+)_{\be}\lra
(U_{r,s}^+)_{\be-\al_i},$$
such that
\begin{gather*}
\Delta(x)=x\otimes 1+\sum_{i\in I}\hat\pa_i(x)\,\om_i\otimes
e_i+\text{the rest},\\
\Delta(x)=\om_\be\otimes x+\sum_{i\in
I}e_i\,\om_{\be-\al_i}\otimes\,_i{\hat\pa}(x)+\text{the rest}.
\end{gather*}
Here in each case ``the rest" refers to terms involving products of
more than one $e_j$ in the second (resp. first) factor. Let
\begin{equation}
\pa_i:=\frac{1}{r_i-s_i}\hat\pa_i,\quad {_i{\pa}}:=\frac{1}{r_i-s_i}{_i\hat{\pa}}.
\end{equation}
Then we have the following lemma

\begin{lem}\label{sk} For any $i,\, j\in I, \, x\in (U_{r,s})_{\be}^+, \, x'\in (U_{r,s}^+)_{\be'}$, $y\in
U_{r,s}^-$, we have

$(\text{\rm i})$ \ \;
$\partial_i(xx')=r^{\lan\be',\al_i\ran}s^{-\lan\al_i,\be'\ran}\pa_i(x)\,x'+x\,\partial_i(x')$,

$(\text{\rm ii})$ \;
$_i\partial(xx')=\,_i\partial(x)\,x'+r^{\lan\al_i,\be\ran}s^{-\lan\be,\al_i\ran}x\,_i\partial(x')$,

$(\text{\rm iii})$ \ $\lan f_iy,\, x\ran_{r,s} =\lan y,\, _i\partial(x)\ran_{r,s}$,

$(\text{\rm iv})$ \ $\lan yf_i,\, x\ran_{r,s} =\lan y,\, \partial_i(x)\ran_{r,s}$,

$(\text{\rm v})$ \ \,
$f_ix-xf_i=\partial_i(x)\,\om_i-\om_i'\,{_i\pa}(x)$,

$(\text{\rm vi})$ \ \,
$\pa_i\,{_j\pa}={_j\pa}\pa_i$.
\end{lem}

Let ${\e_i},\ {_i\e}: \,U_{r,s}^+\lra U_{r,s}^+$ be defined by
\begin{equation}
\e_i(x)=x\,e_i,\quad {_i\e}\,(x')=e_i\,x',
\end{equation}
for any $x,x'\in (U_{r,s}^+)_{\be}$. Then
\begin{coro}
For any $i,\,j\in I$, we have

$(\text{\rm i})$ \ \;
$\partial_i\,\e_j=r^{\lan j,i\ran}s^{-\lan i,j\ran}\e_j\,\pa_i+\de_{i,j}$,

$(\text{\rm ii})$ \;
${_i\pa}\,{_j\e}=r^{\lan i,j\ran}s^{-\lan j,i\ran}\,{_j\e}\,{_i\pa}+\de_{i,j}$.
\end{coro}

\begin{pro}[Hu-Pei \cite{HP1}]\label{ds} For any $i\neq j$,
\begin{align}
\sum_{k=0}^{1-a_{ij}} (-1)^{k}
\binom{1-a_{ij}}{k}_{r_is_i^{-1}} c^{(k)}_{ij}\,{_i\pa}^{1-a_{ij}-k}\,{_j\pa}\,{_i}\pa^{k}&=0,\label{D1}\\
\sum_{k=0}^{1-a_{ij}} (-1)^{k} \binom{1-a_{ij}}{k}_{r_is_i^{-1}}
c^{(k)}_{ij}\,\pa_i^{k}\,\pa_j\,\pa_i^{1-a_{ij}-k}&=0,\label{D2}
\end{align}
where
$$
c^{(k)}_{ij}=(r_is_i^{-1})^{\frac{k(k-1)}{2}} r^{k\lan
j,i\ran}s^{-k\lan i,j\ran},\qquad\textit{for } \ i\neq j.
$$
\end{pro}

Define two linear operators over $U_{r,s}^+$ as follows:
$$\overline{\pa_i},\  \overline{_i\pa}: \; U_{r,s}^+\longrightarrow U_{r,s}^+$$
by
\begin{eqnarray*}
\overline{\pa_i}(x)&=&r^{-\lan\al_i,\be-\al_i\ran}s^{\lan\be-\al_i,\al_i\ran}{_i\pa(x)},\ \quad \forall\; x\in (U_{r,s}^+)_{\be},\\
\overline{_i\pa}(x)&=&r^{-\lan\be-\al_i,\al_i\ran}s^{\lan\al_i,\be-\al_i\ran}\pa_i(x),\
\quad \forall\; x\in (U_{r,s}^+)_{\be}.
\end{eqnarray*}
\begin{lem} For any $i,j\in I$, $x\in(U_{r,s}^+)_{\be}, \; x'\in(U_{r,s}^+)_{\be'}$, we have
\begin{align*}
\overline{\pa_i}(xx')&=\,r^{-\lan\al_i,\be' \ran}s^{\lan\be',\al_i\ran}\overline{\pa_i}(x)\,x'+x\,\overline{\pa_i}(x'),\\
\overline{_i\pa}(xx')&=\,\overline{_i\pa}(x)\,x'+r^{-\lan\be,\al_i\ran}s^{\lan\al_i,\be\ran}x\,\overline{_i\pa}(x'),\\
\overline{\pa_i}\e_j&=r^{-\lan i,j\ran}s^{\lan j,i\ran}\e_j\,\overline{\pa_i}+\de_{i,j},\\
\overline{_i\pa}\,{_j\e}&=r^{-\lan j,i\ran}s^{\lan i,j\ran}{_j\e}\,\overline{_i\pa}+\de_{i,j}.
\end{align*}
\end{lem}
\begin{pro} For any $i\neq j$, we have
\begin{align}
&\sum_{k=0}^{1-a_{ij}} (-1)^{k}
\binom{1-a_{ij}}{k}_{r_is_i^{-1}} c^{(k)}_{ij}\,\overline{\pa_i}^{k}\,\overline{\pa_j}\,\overline{\pa_i}^{1-a_{ij}-k}=0,\label{dsk1}\\
&\sum_{k=0}^{1-a_{ij}} (-1)^{k} \binom{1-a_{ij}}{k}_{r_is_i^{-1}}
c^{(k)}_{ij}\,\overline{_i\pa}^{1-a_{ij}-k}\,\overline{_j\pa}\,\overline{_i\pa}^{k}=0,\label{dsk2}
\end{align}
where
$$
c^{(k)}_{ij}=(r_is_i^{-1})^{\frac{k(k-1)}{2}} r^{k\lan
j,i\ran}s^{-k\lan i,j\ran},\qquad\textit{for } \ i\neq j.
$$
\end{pro}
\begin{proof}
By the definition of $\overline{\pa_i},\  \overline{_i\pa}$ and Proposition \ref{ds}.
\end{proof}

\subsection{Module algebra structure}
For any $\mu\in Q,\; i\in I,\; x\in (U_{r,s}^+)_{\be}$, we define
the action $\triangleright$ by :
\begin{align*}
&\om_{\mu}\triangleright x=r^{\lan\be,\mu\ran}s^{-\lan\mu,\be\ran}x,\\
&\om_{\mu}'\triangleright x=r^{-\lan\mu,\be\ran}s^{\lan\be ,\mu\ran}x,\\
&e_i\triangleright x=\frac{e_ix-r^{\lan\be,\al_i\ran}s^{-\lan\al_i,\be\ran}xe_i}{r_i-s_i},\\
&f_i\triangleright x=\overline{\pa_i}(x).
\end{align*}

\begin{theo}
$U_{r,s}^+$ is a $U_{r,s}(\mg)$-module algebra.
\end{theo}
\begin{proof}
First, we show that $U_{r,s}^+$ is a $U_{r,s}(\mg)$-module. In fact,
for $x\in (U_{r,s}^+)_\beta$,
\begin{align*}
(r_i-s_i)&(e_i\triangleright f_j\triangleright-f_j\triangleright e_i\triangleright)(x)\\
=\ &(r_i-s_i)e_i\triangleright(\overline{\pa_j}(x))-f_j\triangleright(e_ix-r^{\lan\be,\al_i\ran}s^{-\lan\al_i,\be\ran}xe_i)\\
\end{align*}
\begin{align*}
=\
&e_i\overline{\pa_j}(x)-r^{\lan\be-\al_j,\al_i\ran}s^{-\lan\al_i,\be-\al_j\ran}\overline{\pa_j}(x)e_i-
\overline{\pa_j}(e_ix)+r^{\lan\be,\al_i\ran}s^{-\lan\al_i,\be\ran}\overline{\pa_j}(xe_i)\\
=\
&e_i\overline{\pa_j}(x)-r^{\lan\be-\al_j,\al_i\ran}s^{-\lan\al_i,\be-\al_j\ran}\overline{\pa_j}(x)e_i-
(\de_{i,j}r^{-\lan\al_j,\be\ran}s^{\lan\be,\al_j\ran}x
+e_i\overline{\pa_j}(x))\\
&+r^{\lan\be,\al_i\ran}s^{-\lan\al_i,\be\ran}(\de_{i,j}x+r^{-\lan j,i\ran}s^{\lan i,j\ran}\overline{\pa_j}(x)e_i)\\
=\ &\de_{i,j}(r^{\lan\be,\al_i\ran}s^{-\lan\al_i,\be\ran}-r^{-\lan\al_j,\be\ran}s^{\lan\be,\al_j\ran})(x)\\
=\ &\de_{i,j}(\om_i-\om_j')\triangleright x.
\end{align*}
It is clear that $e_i$ satisfies $(r,s)$-Serre relation (R6). By (\ref{dsk1}),
$f_i$ satisfies $(r,s)$-Serre relation (R7).

Next, we show that $U_{r,s}^+$ is a $U_{r,s}(\mg)$-module algebra.
For any $x\in(U_{r,s}^+)_{\be},x'\in(U_{r,s}^+)_{\be'}$,
\begin{gather*}
\om_{\mu}\triangleright(xx')=r^{\lan\be+\be',\mu\ran}s^{-\lan\mu,\be+\be'\ran}xx'
=(\om_{\mu}\triangleright
x)(\om_{\mu}\triangleright x'),
\\
\om_{\mu}'\triangleright(xx')=r^{-\lan\mu,\be+\be'\ran}s^{\lan\be+\be',\mu\ran}xx'
=(\om_{\mu}'\triangleright
x)(\om_{\mu}'\triangleright x'),
\end{gather*}
\begin{align*}
(r_i-s_i)&e_i\triangleright(xx')\\
=&\ e_ixx'-r^{\lan\be+\be',\al_i\ran}s^{-\lan\al_i,\be+\be'\ran}xx'e_i\\
=&\ (e_ix-r^{\lan\be,\al_i\ran}s^{-\lan\al_i,\be\ran}xe_i)x'
+r^{\lan\be,\al_i\ran}s^{-\lan\al_i,\be\ran}x(e_ix'-r^{\lan\be',\al_i\ran}s^{-\lan\al_i,\be'\ran}x'e_i)\\
=&\ (r_i-s_i)(e_i\triangleright x)(1\triangleright x')+(r_i-s_i)(\om_{i}\triangleright x)(e_i\triangleright x'),
\end{align*}
\begin{align*}
f_i\triangleright(xx')=&\,\overline{\pa_i}(xx')
= r^{-\lan\al_i,\be'\ran}s^{\lan\be'\al_i\ran}\overline{\pa_i}(x)\,x'+x\,\overline{\pa_i}(x')\\
=&\, (f_i\triangleright x)(\om_i'\triangleright x')+(1\triangleright
x)(f_i\triangleright x').
\end{align*}

This completes the proof.
\end{proof}

\section{\bf Hopf $2$-cocycle deformation}

\begin{defi}\label{HC}
Let $(H,m,1,\De,\vep,S)$ be a Hopf algebra over a field $k$. A
bilinear form $\si: \; H\times H \longrightarrow k$ is called a Hopf
$2$-cocycle for $H$ if it has an inverse $\si^{-1}$ under the
convolution product, and satisfies the cocycle conditions:
\begin{gather}
\sum\si(a_1, b_1)\si(a_2b_2, c) =\sum\si(b_1, c_1)\si(a, b_2c_2), \label{cocy1}\\
\si(a, 1) =\vep(a) = \si(1, a), \qquad\text{$\forall\; a, b, c\in
H$}.\label{cocy2}
\end{gather}

\end{defi}
Following Doi-Takeuchi \cite{DT} and Majid \cite{Ma}, one can construct a new Hopf
algebra structure $(H^{\si}$, $m^{\si}, 1,\Delta, \vep, S^{\si})$ on
the coalgebra $(H,\Delta,\vep)$. The new multiplication $m^{\si}$ is
given by
\begin{eqnarray}
m^{\si}(a\ot b)=\sum\si(a_1, b_1)a_2b_2\si^{-1}(a_3, b_3).
\end{eqnarray}
The new antipode $S^{\si}$ is given by
\begin{eqnarray}
S^{\si}(a)=\sum
\si^{-1}(a_1,S(a_2))S(a_3)\si(S(a_4),a_5),\quad\forall\; a\in H.
\end{eqnarray}

\begin{pro}
Assume that $rs^{-1}=q^2$. Let $\si:U_{q}(\mg)\times U_{q}(\mg)\lra
\bK$ be a bilinear map defined as
\begin{equation}
\si(x,y)=\begin{cases}
r^{\frac{\lan\nu,\mu\ran}{2}}s^{-\frac{\lan\mu,\nu\ran}{2}},\qquad\forall\ \
x=K_{\mu},\ \text{or}\ K_{\mu}',\ \  y=K_{\nu},\
\text{or}\ K_{\nu}',\ \mu,\nu\in Q,\\
0,\qquad\qquad\qquad\quad\text{otherwise}.
\end{cases}\label{cocy3}
\end{equation}
Then $\si$ is a Hopf $2$-cocycle on $U_{q}$ in the sense of
Definition \ref{HC}.
\end{pro}

\begin{proof} It is clear that $\si$ satisfies the condition (\ref{cocy2}).
For any homogeneous elements $x,\, y,\, z\in U_{q}$, if $x,\, y,\,
z\in U_{q}^0$, it is clear that $\si$ satisfies the condition
(\ref{cocy1}). If $x\notin U_{q}^0$, then $\sum\si(y_1, z_1)\si(x,
y_2z_2)=0$. On the other hand, let $a\ot b$ be any one of the
summands in $\De(x)$, then $a\ot b\notin U_{q}^0\ot U_{q}^0$, that
is, $a\notin U_{q}^0$ or $b\notin U_{q}^0$. It follows that
$\sum\si(a, y_1)\si(by_2, z)=0$. Therefore, $\si$ also satisfies the
condition (\ref{cocy1}). Similarly, we can prove that $\si$ also
satisfies the condition (\ref{cocy1}) if $y\notin U_{q}^0$ or
$z\notin U_{q}^0$.
\end{proof}

\begin{theo}\label{Thm1}
Let $rs^{-1}=q^2$ and $\si$ be defined as (\ref{cocy3}). Then, as
$\bK$-Hopf algebras,
$$
U_{r,s}(\mg)\simeq U_{q}^{\si}(\mg),
$$ where $U_{q}^{\si}$ is the
new Hopf algebra arising from the deformation by Hopf $2$-cocycle
$\si$.
\end{theo}
\begin{proof}
Let $\phi: U_{r,s}(\mg)\lra U_{q}^\si(\mg)$ be a linear map defined by
$$
\phi(e_i)=E_i, \quad \phi(f_i)=(s_iq_i)^{-1}F_i, \quad
\phi(\om_i)=K_i, \quad \phi(\om_i')=K_i'.
$$
It suffices to show that $\phi$ is an algebra homomorphism. Note
that
\begin{gather*}
\De^2(K_i)=K_i\ot K_i\ot K_i, \quad \De^2(K_i')=K_i'\ot K_i'\ot K_i',\\
\De^2(E_i)=E_i\ot 1\ot1+K_i\ot E_i\ot 1+K_i\ot K_i\ot E_i,\\
\De^2(F_i)= 1\ot1\ot F_i+1\ot F_i\ot K_i'+F_i\ot K_i'\ot K_i'.
\end{gather*}
Define $x*y:=m^{\si}(x\ot y),\quad\forall\  x,\, y\in U_{q}$. Then
it is easy to check
\begin{gather*}
K_i^{\pm1}*K_i^{\mp1}=K_i'^{\pm1}*K_i'^{\mp1}=1,\\
K_i*K_j=K_j*K_i,\quad K_i'*K_j'=K_j'*K_i',\quad K_i*K_j'=K_j'*K_i.
\end{gather*}
By definition, we have
\begin{eqnarray*}
K_i*E_j&=&\si(K_i,K_j)K_iE_j=\si(K_i,K_j)q^{\lan i,j\ran+\lan j,i\ran}E_jK_i\\
&=&\si(K_i,K_j)q^{\lan i,j\ran+\lan j,i\ran}\si(K_j,K_i)^{-1}E_j*K_i\\
&=&r^{\frac{\lan j,i\ran}{2}}s^{-\frac{\lan i,
j\ran}{2}}(rs^{-1})^{\frac{\lan i,j\ran+\lan j,i\ran}{2}}
r^{-\frac{\lan i,j\ran}{2}}s^{\frac{\lan j, i\ran}{2}}E_j*K_i\\
&=&r^{\lan j, i\ran}s^{-\lan i, j\ran}E_j*K_i,
\\
K_i'*E_j&=&\si(K_i',K_j)K_i'E_j=\si(K_i',K_j)q^{-\lan i,j\ran-\lan j,i\ran}E_jK_i'\\
&=&\si(K_i',K_j)q^{-\lan i,j\ran-\lan j,i\ran}\si(K_j,K_i')^{-1}E_j*K_i'\\
&=&r^{\frac{\lan j,i\ran}{2}}s^{-\frac{\lan i,
j\ran}{2}}(rs^{-1})^{-\frac{\lan i,j\ran+\lan j,i\ran}{2}}
r^{-\frac{\lan i,j\ran}{2}}s^{\frac{\lan j, i\ran}{2}}E_j*K_i'\\
&=&r^{-\lan i, j\ran}s^{\lan j, i\ran}E_j*K_i',
\\
K_i*F_j&=&\si(K_i,K_j)^{-1}K_iF_j=\si(K_i,K_j)^{-1}q^{-\lan i,j\ran-\lan j,i\ran}F_jK_i\\
&=&\si(K_i,K_j)^{-1}q^{-\lan i,j\ran-\lan j,i\ran}\si(K_j,K_i)F_j*K_i\\
&=&r^{-\frac{\lan j,i\ran}{2}}s^{\frac{\lan i,
j\ran}{2}}(rs^{-1})^{-\frac{\lan i,j\ran+\lan j,i\ran}{2}}
r^{\frac{\lan i,j\ran}{2}}s^{-\frac{\lan j, i\ran}{2}}F_j*K_i\\
&=&r^{-\lan j, i\ran}s^{\lan i, j\ran}F_j*K_i,
\\
K_i'*F_j&=&\si(K_i',K_j)^{-1}K_i'F_j=\si(K_i',K_j)^{-1}q^{\lan i,j\ran+\lan j,i\ran}F_jK_i'\\
&=&\si(K_i',K_j)^{-1}q^{\lan i,j\ran+\lan j,i\ran}\si(K_j,K_i')F_j*K_i'\\
&=&r^{-\frac{\lan j,i\ran}{2}}s^{\frac{\lan i,
j\ran}{2}}(rs^{-1})^{\frac{\lan i,j\ran+\lan j,i\ran}{2}}
r^{\frac{\lan i,j\ran}{2}}s^{-\frac{\lan j, i\ran}{2}}F_j*K_i'\\
&=&r^{\lan i, j\ran}s^{-\lan j, i\ran}F_j*K_i',
\\
E_i*F_j-F_j*E_i&=&E_iF_j-F_jE_i=\delta_{i,j}\frac{K_i-K_i'}{q_i-q_i^{-1}}.
\end{eqnarray*}
Next, we will prove
\begin{gather*}
\sum_{k=0}^{1-a_{ij}}(-1)^{k} \binom{1-a_{ij}}{k}_{r_is_i^{-1}}
c^{(k)}_{ij} E_{i}^{*(1-a_{ij}-k)}*E_{j}* E_{i}^{*k} =0, \  (i\neq
j),\\
\sum_{k=0}^{1-a_{ij}}(-1)^{k} \binom{1-a_{ij}}{k}_{r_is_i^{-1}}
c^{(k)}_{ij}\, F_{i}^{*k}*F_{j}*F_{i}^{*(1-a_{ij}-k)} =0, \ (i\ne
j),
\end{gather*} where
$c^{(k)}_{ij}=(r_is_i^{-1})^{\frac{k(k-1)}{2}} r^{k\lan
j,i\ran}s^{-k\lan i,j\ran}$.
Since
\begin{equation*}
\begin{split}
E_{i}^{*(1-a_{ij}-k)}{*}E_{j}{*}E_{i}^{*k}
&=\si(K_i,K_i)^{\frac{(a_{ij}-1)a_{ij}}{2}}\si(K_i,K_j)^{1-a_{ij}-k}\si(K_j,K_i)^{k}E_{i}^{1-a_{ij}-k} E_{j}E_{i}^{k}\\
&=(r_is_i^{-1})^{\frac{(a_{ij}-1)a_{ij}}{2}}r^{\frac{(1-a_{ij}-k)\lan
j,i\ran+k\lan i,j\ran}{2}}
s^{-\frac{(1-a_{ij}-k)\lan i,j\ran+k\lan j,i\ran}{2}}E_{i}^{1-a_{ij}-k} E_{j}E_{i}^{k}\\
\end{split}
\end{equation*}
\begin{equation*}
\begin{split}
&=(r_is_i^{-1})^{\frac{(a_{ij}-1)a_{ij}}{2}} (r^{\lan
j,i\ran}s^{-\lan i,j\ran})^{\frac{1-a_{ij}}{2}} (rs)^{\frac{-k\lan
j,i\ran+k\lan i,j\ran}{2}}E_{i}^{1-a_{ij}-k} E_{j}E_{i}^{k},
\\
c^{(k)}_{ij}(rs)^{\frac{-k\lan j,i\ran+k\lan i,j\ran}{2}}
&=(r_is_i^{-1})^{\frac{k(k-1)}{2}}(r_is_i^{-1})^{\frac{ka_{ij}}{2}}
=q_i^{k(k-1+a_{ij})},
\end{split}
\end{equation*}
then
\begin{equation*}
\begin{split}
&\sum_{k=0}^{1-a_{ij}} (-1)^{k} \binom{1-a_{ij}}{k}_{r_is_i^{-1}}
c^{(k)}_{ij} E_{i}^{*(1-a_{ij}-k)}* E_{j}*
E_{i}^{*k}\\
&\quad\,=(r_is_i^{-1})^{\frac{(a_{ij}-1)a_{ij}}{2}} (r^{\lan
j,i\ran}s^{-\lan i,j\ran})^{\frac{1-a_{ij}}{2}}\sum_{k=0}^{1-a_{ij}}
(-1)^{k} \binom{1-a_{ij}}{k}_{q_i^2}
q_i^{{k(k-1+a_{ij})}}E_{i}^{1-a_{ij}-k}
E_{j}E_{i}^{k}\\
&\quad\,=(r_is_i^{-1})^{\frac{(a_{ij}-1)a_{ij}}{2}} (r^{\lan
j,i\ran}s^{-\lan i,j\ran})^{\frac{1-a_{ij}}{2}}\sum_{k=0}^{1-a_{ij}}
(-1)^{k}
 {\left[{1-a_{ij}}\atop{k}\right]}_{q_i}E_{i}^{1-a_{ij}-k}
E_{j}E_{i}^{k}\\
&\quad\,=0.
\end{split}
\end{equation*}
Since
\begin{equation*}
\begin{split}
F_{i}^{*k}{*}F_{j}{*}F_{i}^{*(1-a_{ij}-k)}
&=\si(K_i',K_i')^{-\frac{(a_{ij}-1)a_{ij}}{2}}\si(K_i',K_j')^{-k}\si(K_j',K_i')^{-(1-a_{ij}-k)}
F_{i}^{k} F_{j}F_{i}^{1-a_{ij}-k}\\
&=(r_is_i^{-1})^{-\frac{(a_{ij}-1)a_{ij}}{2}}r^{-\frac{k\lan j,i\ran
+(1-a_{ij}-k)\lan i,j\ran}{2}}
s^{\frac{k\lan i,j\ran+(1-a_{ij}-k)\lan j,i\ran}{2}}F_{i}^{k}F_{j}F_{i}^{1-a_{ij}-k} \\
&=(r_is_i^{-1})^{-\frac{(a_{ij}-1)a_{ij}}{2}}(r^{\lan
i,j\ran}s^{-\lan j,i\ran})^{-\frac{1-a_{ij}}{2}} (rs)^{-\frac{k\lan
j,i\ran-k\lan i,j\ran}{2}} F_{i}^{k}F_{j}F_{i}^{1-a_{ij}-k},\\
c^{(1-a_{ij}-k)}_{ij}(rs)&^{\frac{-(1-a_{ij}-k)\lan j,i\ran+(1-a_{ij}-k)\lan i,j\ran}{2}}\\
&=(r_is_i^{-1})^{\frac{(a_{ij}-1+k)(a_{ij}+k)}{2}}(r_is_i^{-1})^{\frac{(1-a_{ij}-k)a_{ij}}{2}}
=q_i^{k(k-1+a_{ij})}.
\end{split}
\end{equation*}
Then
\begin{equation*}
\begin{split}
\sum_{k=0}^{1-a_{ij}}& (-1)^{k} \binom{1-a_{ij}}{k}_{r_is_i^{-1}}
c^{(k)}_{ij} F_{i}^{*k}* F_{j}*
F_{i}^{*(1-a_{ij}-k)}\\
&=(r_is_i^{-1})^{-\frac{(a_{ij}-1)a_{ij}}{2}}
(r^{\lan i,j\ran}s^{-\lan j,i\ran})^{-\frac{1-a_{ij}}{2}}\times\\
&\quad\sum_{k=0}^{1-a_{ij}} (-1)^{k}
\binom{1-a_{ij}}{k}_{q_i^2}c^{(k)}_{ij}(rs)^{-\frac{k\lan
j,i\ran-k\lan i,j\ran}{2}}
F_{i}^{k}F_{j}F_{i}^{1-a_{ij}-k}\\
&=(r_is_i^{-1})^{-\frac{(a_{ij}-1)a_{ij}}{2}}
(r^{\lan i,j\ran}s^{-\lan j,i\ran})^{-\frac{1-a_{ij}}{2}}\times\\
&\quad\sum_{k=0}^{1-a_{ij}} (-1)^{k}
\binom{1-a_{ij}}{k}_{q_i^2}c^{(k)}_{ij}(rs)^{-\frac{k\lan
j,i\ran-k\lan i,j\ran}{2}}
F_{i}^{k}F_{j}F_{i}^{1-a_{ij}-k}\\
\end{split}
\end{equation*}
\begin{equation*}
\begin{split}&=(r_is_i^{-1})^{-\frac{(a_{ij}-1)a_{ij}}{2}}
(r^{\lan i,j\ran}s^{-\lan j,i\ran})^{-\frac{1-a_{ij}}{2}}\times\\
&\quad \sum_{k=0}^{1-a_{ij}} (-1)^{1-a_{ij}-k}
\binom{1-a_{ij}}{k}_{q_i^2}c^{(1-a_{ij}-k)}_{ij}(rs)^{-\frac{(1-a_{ij}-k)\lan
j,i\ran-(1-a_{ij}-k)\lan i,j\ran}{2}}
F_{i}^{1-a_{ij}-k}F_{j}F_{i}^k\\
&=(r_is_i^{-1})^{-\frac{(a_{ij}-1)a_{ij}}{2}} (r^{\lan
i,j\ran}s^{-\lan
j,i\ran})^{-\frac{1-a_{ij}}{2}}\sum_{k=0}^{1-a_{ij}} (-1)^{k}
\binom{1-a_{ij}}{k}_{q_i^2}q_i^{{k(k-1+a_{ij})}}
F_{i}^{1-a_{ij}-k}F_{j}F_{i}^k\\
&=(r_is_i^{-1})^{-\frac{(a_{ij}-1)a_{ij}}{2}} (r^{\lan
i,j\ran}s^{-\lan
j,i\ran})^{-\frac{1-a_{ij}}{2}}\sum_{k=0}^{1-a_{ij}} (-1)^{k}
 {\left[{1-a_{ij}}\atop{k}\right]}_{q_i}F_{i}^{1-a_{ij}-k}
F_{j}F_{i}^k\\
&=0.
\end{split}
\end{equation*}
\end{proof}

\begin{rem}
Let $rs^{-1}=q^2$ and $\si':U_{q}(\mg)\times U_{q}(\mg)\lra \bK$ be
a bilinear form defined as
\begin{equation*}
\si'(x,y)=\begin{cases}
(\frac{q}{r})^{\lan\mu,\nu\ran},\qquad\forall\ \
x=K_{\mu},\ \text{or}\ K_{\mu}',\ \  y=K_{\nu},\
\text{or}\ K_{\nu}',\ \mu,\nu\in Q,\\
0,\qquad\qquad\qquad\quad\text{otherwise}.
\end{cases}
\end{equation*}
Similarly, one can check directly that $\si'$ is also a Hopf
$2$-cocycle of $U_{q}(\mg)$. In fact, $\si$ and $\si'$ are
cohomologous Hopf $2$-cocycles in the sense of Majid \cite{Ma}.
\end{rem}

\section{\bf Bigraded Hopf algebras}

\begin{lem}[\cite{AST}] \ Let $A=\bigoplus_{g\in G}A_g$ be a $G$-graded associative
algebra over a field $k$,
where $G$ is an abelian group. Let $\psi: G\times G\to k^* $ be a
$2$-cocycle of group $G$. We introduce a new multiplication $\cdot$ on
$A$ as follows: For any $x\in A_g, \,y\in A_h$, where $g,\,h\in
G$, we define
$$
x\cdot y=\psi(g,h)\,x\,y.
$$
Denote this new algebra by $A^{\psi}$. Then $A^{\psi}$ is a
$G$-graded associative algebra.
\end{lem}

\begin{defi}[\cite{HLT}]
Let $G=\{g_i\;|\;i\in I\}$ be a free abelian group. A Hopf algebra
$(A,i,m,\vep,\Delta,S)$ over a field $k$ is a $G$-bigraded Hopf
algebra if it is equipped with a $G\times G$-grading
$$
H=\sum_{(\al,\be)\in G\times G} H_{\al,\be}
$$
such that
\begin{equation*}
\begin{split}
k&\subseteq H_{0,0},\\
H_{\al,\be}H_{\al',\be'}&\subset H_{\al+\al',\be+\be'},\\
\end{split}
\end{equation*}
\begin{equation*}
\begin{split}\De(H_{\al,\be})&\subset\sum_{\ga\in G}H_{\al,\ga}\ot H_{-\ga,\be},\\
\varepsilon(H_{\al,\be})&=0,\quad \text{for } \  \al\neq -\be, \\
S(H_{\al,\be})&\subset H_{\be,\al}.
\end{split}
\end{equation*}
\end{defi}
Let $\si: G\times G\to k^{\ast}$ be a skew bicharacter over $G$.
Then one can define $\tilde{\si}:(G\times G)\times(G\times G)\lra
k^{\ast}$ such that
$$
\tilde{\si}((\al,\be),(\al',\be'))=\si(\al,\al')\si(\be,\be')^{-1}.
$$
It is clear that $\tilde{\si}$ is a $2$-cocycle of group $G\times
G$. Let $(H,1,m,\vep,\Delta,S)$ be a $G$-bigraded Hopf algebra.
Define a new product $\circ$ as
\begin{equation*}
a\circ
b=\tilde{\si}((\al,\be),(\al',\be'))ab=\si(\al,\al')\si(\be,\be')^{-1}ab,\quad\forall\
a\in A_{\al,\be}, \ b\in A_{\al',\be'}.
\end{equation*}
Then $H_\si:=(H,1,\circ,\vep,\Delta,S)$ is a $G$-bigraded Hopf algebra \cite{HLT}.

\subsection{}
For any $i\in I$ and $\mu\in Q$, define
$$ E_i\in (U_{q})_{\al_i,0},\ F_i\in (U_{q})_{0,-\al_i},\   K_i\in (U_{q})_{\al_i,-\al_i},\
K_i'\in (U_{q})_{\al_i,-\al_i}.$$
It is clear that $U_{q}(\mg)$ is a $Q$-bigraded Hopf algebra.

Assume that $r_is_i^{-1}=q^{2d_i}$, $\forall\ i\in I$. Let
\begin{equation}
p_{ij}=r^{\lan j, i\ran}s^{-\lan i, j\ran}q^{-d_ia_{ij}}.
\end{equation}
Then
\begin{equation}
p_{ii}=1,\quad p_{ij}p_{ji}=1,\quad\forall\; i, \,j\in I.
\end{equation}

\begin{pro}\label{bigrading}
Let $\zeta:Q\times Q\to \bK^*,\
\zeta(\al_i,\al_j)=p_{ij}^{\frac{1}{2}},\ \forall \,i,\,j\in I$.
Then $\zeta$ is a skew bicharacter on $Q$. As $Q$-bigraded Hopf
algebras, we have
\begin{equation}
U_{q,\zeta}(\mg)\simeq U_{r,s}(\mg).
\end{equation}
\end{pro}
\begin{proof} It suffices to check the following relations:
\begin{align*}
K_i\circ E_j&=\zeta(\al_i,\al_j)\zeta^{-1}(-\al_i,0)K_iE_j\\
&=\zeta(\al_i,\al_j)q^{d_ia_{ij}}\zeta(\al_j,\al_i)^{-1}E_j\circ K_i\\
\end{align*}
\begin{align*}
&=p_{ij}q^{d_ia_{ij}}E_j\circ K_i\\
&=r^{\lan j, i\ran}s^{-\lan i, j\ran}E_j\circ K_i,
\\
K_i'\circ E_j&= \zeta(\al_i,\al_j)\zeta^{-1}(-\al_i,0)K_i'E_j\\
&=\zeta(\al_i,\al_j)q^{-d_ia_{ij}}\zeta(\al_j,\al_i)^{-1}E_j\circ K_i'\\
&=p_{ij}q^{-d_ia_{ij}}E_j\circ K_i'\\
&=r^{-\lan i, j\ran}s^{\lan j, i\ran}E_j\circ K_i,
\\
K_i\circ F_j&=\zeta(\al_i,0)\zeta^{-1}(-\al_i,-\al_j)K_iF_j\\
&=\zeta^{-1}(\al_i,\al_j)q^{-d_ia_{ij}}\zeta(\al_j,\al_i)F_j\circ K_i'\\
&=p_{ij}^{-1}q^{-d_ia_{ij}}F_j\circ K_i'\\
&=r^{-\lan j, i\ran}s^{\lan i, j\ran}F_j\circ K_i',
\\
K_i'\circ F_j&=\zeta(\al_i,0)\zeta^{-1}(-\al_i,-\al_j)K_i'F_j\\
&=\zeta^{-1}(\al_i,\al_j)q^{d_ia_{ij}}\zeta(\al_j,\al_i)F_j\circ K_i'\\
&=p_{ji}q^{d_ia_{ij}}F_j\circ K_i'\\
&=r^{\lan i, j\ran}s^{-\lan j, i\ran}F_j\circ K_i',
\\
E_i\circ F_j-F_j\circ
E_i&=E_iF_j-F_jE_i=\de_{i,j}\frac{K_i-K_i'}{q^{d_i}-q^{-d_i}}.
\end{align*}
Since
$$
E_{i}^{\circ(1-a_{ij}-k)}\circ E_{j}\circ E_{i}^{\circ k} =p_{ij}^{\frac{1-a_{ij}}{2}}p_{ij}^{-k}
E_{i}^{1-a_{ij}-k} E_{j}E_{i}^{k},
$$
\begin{equation*}
\begin{split}
\sum_{k=0}^{1-a_{ij}}& (-1)^{k} \binom{1-a_{ij}}{k}_{r_is_i^{-1}}
c^{(k)}_{ij} E_{i}^{\circ(1-a_{ij}-k)}\circ E_{j}\circ
E_{i}^{\circ k} \\
&=\sum_{k=0}^{1-a_{ij}} (-1)^{k} \binom{1-a_{ij}}{k}_{r_is_i^{-1}}
(r_is_i^{-1})^{\frac{k(k-1)}{2}} r^{k\lan j,i\ran}s^{-k\lan i,j\ran}
p_{ij}^{\frac{1-a_{ij}}{2}}p_{ij}^{-k} E_{i}^{1-a_{ij}-k}
E_{j}E_{i}^{k} \\
&=p_{ij}^{\frac{1-a_{ij}}{2}}\sum_{k=0}^{1-a_{ij}}
(-1)^{k} \binom{1-a_{ij}}{k}_{q^{2d_i}} q^{d_ik(k-1)}r^{k\lan
j,i\ran}s^{-k\lan i,j\ran}p_{ij}^{-k} E_{i}^{1-a_{ij}-k}
E_{j}E_{i}^{k}
\\
&=p_{ij}^{\frac{1-a_{ij}}{2}}\sum_{k=0}^{1-a_{ij}} (-1)^{k}
\binom{1-a_{ij}}{k}_{q^{2d_i}} q^{d_ik(k-1+a_{ij})}
E_{i}^{1-a_{ij}-k} E_{j}E_{i}^{k}
\end{split}
\end{equation*}
\begin{equation*}
\begin{split}
&=p_{ij}^{\frac{1-a_{ij}}{2}} \sum_{k=0}^{1-a_{ij}} (-1)^{k}
{\left[{1-a_{ij}}\atop{k}\right]}_{q_i}E_{i}^{1-a_{ij}-k}
E_{j}E_{i}^{k}\\
&=0.
\end{split}
\end{equation*}
Similarly,
$$
F_{i}^{\circ k}\circ F_{j}\circ F_{i}^{\circ(1-a_{ij}-k)}=
p_{ij}^{\frac{1-a_{ij}}{2}}p_{ij}^{-k}
F_{i}^{k}F_{j}F_{i}^{1-a_{ij}-k}.
$$
Then
\begin{equation*}
\begin{split}
\sum_{k=0}^{1-a_{ij}}& (-1)^{k} \binom{1-a_{ij}}{k}_{r_is_i^{-1}}
c^{(k)}_{ij}\,
F_{i}^{\circ k}\circ F_{j}\circ F_{i}^{\circ(1-a_{ij}-k)}\\
&=\sum_{k=0}^{1-a_{ij}} (-1)^{k}
\binom{1-a_{ij}}{k}_{r_is_i^{-1}}(r_is_i^{-1})^{\frac{k(k-1)}{2}}
r^{k\lan j,i\ran}s^{-k\lan
i,j\ran}p_{ij}^{\frac{1-a_{ij}}{2}}p_{ij}^{-k} F_{i}^{k}
F_{j}F_{i}^{1-a_{ij}-k}\\
&=p_{ij}^{\frac{1-a_{ij}}{2}}\sum_{k=0}^{1-a_{ij}} (-1)^{k}
\binom{1-a_{ij}}{k}_{q^{2d_i}} q^{d_ik(k-1+a_{ij})} F_{i}^{k}
F_{j}F_{i}^{1-a_{ij}-k} \\
&=p_{ij}^{\frac{1-a_{ij}}{2}} \sum_{k=0}^{1-a_{ij}} (-1)^{k}
{\left[{1-a_{ij}}\atop{k}\right]}_{q_i}F_{i}^{k}
F_{j}F_{i}^{1-a_{ij}-k}\\
&=p_{ij}^{\frac{1-a_{ij}}{2}} \sum_{k=0}^{1-a_{ij}} (-1)^{k}
{\left[{1-a_{ij}}\atop{k}\right]}_{q_i}F_{i}^{1-a_{ij}-k}
F_{j}F_{i}^{k}\\
&=0.
\end{split}
\end{equation*}

This completes the proof.
\end{proof}

As a corollary, we recover a result in \cite{HP1}.
\begin{coro}[\cite{HP1}]\label{deform2}
Let $U_{q,\zeta}^{\pm}$ be the deformation of $U_{q}^{\pm}$ by $\zeta$.
Then, as $Q$-graded algebras, we have
$$
U_{r,s}^{\pm}\simeq U_{q,\zeta}^{\pm}.
$$
\end{coro}

\begin{rem}
When $\mg$ is of type $D_n$,  $U_{r,s}(D_n)$ is slightly different
from $U_{r,s}'(D_n)$ defined in \cite{BGH1}. In fact, they are
related by a $Q$-bigraded deformation.

For example, for type $D_4$, we have
$$
U_{r,s}(D_4)\to\left(\begin{array}{cccc}
 rs^{-1} & r^{-1}&1&1\\
 s  & rs^{-1} & r^{-1} &r^{-1}\\
1& s &rs^{-1}&1\\
1& s & 1&rs^{-1} \\
\end{array}\right),\
U_{r,s}'(D_4)\to\left(\begin{array}{cccc}
 rs^{-1} & r^{-1}&1&1\\
  s & rs^{-1} & r^{-1} &r^{-1}\\
1& s &rs^{-1}&(rs)^{-1}\\
1& s & rs&rs^{-1}
\end{array}\right).
$$
It suffices to set
$$
(p_{ij})=\left(\begin{array}{cccc}
 1 & 1&1  &1\\
 1 & 1& 1 &1\\
 1 & 1& 1 &(rs)^{-1}\\
 1 & 1& rs&1
\end{array}\right).
$$
In general, for type $D_n$, it suffices to take
$p_{ij}=(rs)^{\de_{i,n}\de_{j,n-1}-\de_{i,n-1}\de_{j,n}}$.
\end{rem}

\subsection{}
As applications, we will give a new and simple proof for the
existence of non-degenerate skew Hopf pairing on $U_{r,s}(\mg)$.

\begin{pro}\label{deform3}
Let $\lan, \ran_{q}$ be the skew Hopf pairing on $U_{q}(\mg)$.
Define a bilinear form as follows:
$$
\lan\, ,\, \ran_{q,\zeta}:\, U_{q,\zeta}^{\leq0}\times
U_{q,\zeta}^{\geq0}\lra \bK
$$
$$
\lan y,x\ran_{q,\zeta}=\zeta(\be,\al)^{-1}\zeta(\be',\al')^{-1}\lan
y,x \ran_{q}, \quad \forall\; x\in
(U_{q,\zeta}^{\geq0})_{\al,\al'},\ y\in
(U_{q,\zeta}^{\leq0})_{\be,\be'}.
$$
Then $\lan\,,\,\ran_{q,\zeta}$ is a unique nondegenerate skew Hopf
pairing on $U_{r,s}(\mg)$.
\end{pro}
\begin{proof} It is clear that $\lan\,,\,\ran_{q,\zeta}$ is unique and
nondegenerate since $\lan\,,\,\ran_{q}$ is nondegenerate.
\begin{equation*}
\begin{split}
\lan K_{i}',K_{j}\ran_{q,\zeta}&=\zeta(\al_i,\al_j)^{-2}\lan
K_{i}',K_{j}\ran_{q}=p_{ij}^{-1}q^{d_ia_{ij}}\\
&=p_{ji}q^{d_ia_{ij}}=r^{\lan i,j\ran}s^{-\lan
j,i\ran}q^{-d_{i}a_{ij}}q^{d_ia_{ij}}=r^{\lan i,j\ran}s^{-\lan
j,i\ran},
\\
\lan F_i,E_j\ran_{q,\zeta}&=\lan
F_i,E_j\ran_{q}=\de_{i,j}\frac{-1}{q_i-q_i^{-1}}.
\end{split}
\end{equation*}
For any $y\in(U_{q,\zeta}^{\leq0})_{\be,\ga},\
x_i\in(U_{q,\zeta}^{\geq0})_{\be_i,\ga_i}$, we have
\begin{align*}
\lan y,x_1\circ x_2\ran_{q,\zeta}
&=\zeta(\be,\be_1+\be_2)^{-1}\zeta(\ga,\ga_1+\ga_2)^{-1}\zeta(\be_1,\be_2)\zeta(\ga_1,\ga_2)^{-1}\lan y, x_1x_2\ran_{q}\\
&=\zeta(\be,\be_1+\be_2)^{-1}\zeta(\ga,\ga_1+\ga_2)^{-1}\zeta(\be_1,\be_2)\zeta(\ga_1,\ga_2)^{-1}\lan y_1,x_2\ran_{q}\lan y_2,x_1\ran_{q},
\end{align*}
where $\De(y)=\sum y_1\ot y_2, \ y_1\in
(U_{q,\zeta}^{\leq0})_{\be,\nu}, \ y_2\in
(U_{q,\zeta}^{\leq0})_{-\nu,\ga}$ with $\nu\in Q$.

On the other hand,
\begin{align*}
\lan y_1,x_2\ran_{q,\zeta}\lan y_2,x_1\ran_{q,\zeta}
&=\zeta(\be,\be_2)^{-1}\zeta(\nu,\ga_2)^{-1}\lan y_1, x_2\ran_{q}\zeta(-\nu,\be_1)^{-1}\zeta(\ga,\ga_1)^{-1}\lan y_2, x_1\ran_{q}\\
&=\zeta(\be,\be_2)^{-1}\zeta(\nu,\ga_2)^{-1}\zeta(-\nu,\be_1)^{-1}\zeta(\ga,\ga_1)^{-1}\lan y_1, x_2\ran_{q}\lan y_2, x_1\ran_{q}.
\end{align*}
Since $\be+\nu+\be_2+\ga_2=0, \ \ga-\nu+\be_1+\ga_1=0$,
\begin{align*}
&\zeta(\be,\be_2)^{-1}\zeta(\nu,\ga_2)^{-1}\zeta(-\nu,\be_1)^{-1}\zeta(\ga,\ga_1)^{-1}\\
&\quad=\zeta(\be,\be_2)^{-1}\zeta(\be_1+\ga_1+\ga,\ga_2)^{-1}\zeta(\be_2+\ga_2+\be,\be_1)^{-1}\zeta(\ga,\ga_1)^{-1}\\
&\quad=\zeta(\be,\be_1+\be_2)^{-1}\zeta(\ga,\ga_1+\ga_2)^{-1}\zeta(\be_1,\be_2)\zeta(\ga_1,\ga_2)^{-1}.
\end{align*}
Hence, $\lan y,x_1\circ x_2\ran_{q,\zeta}=\lan
y_1,x_2\ran_{q,\zeta}\lan y_2,x_1\ran_{q,\zeta}$.
Similarly,
\begin{align*}
\lan y_1\circ y_2,x\ran_{q,\zeta}
&=\zeta(\be_1+\be_2,\be)^{-1}\zeta(\ga_1+\ga_2,\ga)^{-1}\zeta(\be_1,\be_2)
\zeta(\ga_1,\ga_2)^{-1}\lan y_1y_2,x\ran_{q}\\
&=\zeta(\be_1+\be_2,\be)^{-1}\zeta(\ga_1+\ga_2,\ga)^{-1}\zeta(\be_1,\be_2)
\zeta(\ga_1,\ga_2)^{-1}\lan y_1,x_1\ran_{q}\lan y_2,x_2\ran_{q},
\end{align*}
where $\De(x)=\sum x_1\ot x_2, \ x_1\in
(U_{q,\zeta}^{\geq0})_{\be,\nu}, \ x_2\in
(U_{q,\zeta}^{\geq0})_{-\nu,\ga}$. Since
\begin{align*}
\lan y_1,x_1\ran_{q,\zeta}\lan y_2,x_2\ran_{q,\zeta}
&=\zeta(\be_1,\be)^{-1}\zeta(\ga_1,\nu)^{-1}\zeta(\be_2,-\nu)^{-1}
\zeta(\ga_2,\ga)^{-1}\lan y_1,x_1\ran_{q}\lan y_2,x_2\ran_{q}.
\end{align*}
Since $\be_1+\ga_1+\be+\nu=0, \ \be_2+\ga_2+\ga-\nu=0$,
\begin{align*}
&\zeta(\be_1,\be)^{-1}\zeta(\ga_1,\nu)^{-1}\zeta(\be_2,-\nu)^{-1}\zeta(\ga_2,\ga)^{-1}\\
&\quad= \zeta(\be_1,\be)^{-1}\zeta(\ga_1,\be_2+\ga_2+\ga)^{-1}\zeta(\be_2,\be_1+\ga_1+\be)^{-1}\zeta(\ga_2,\ga)^{-1}\\
&\quad=\zeta(\be_1+\be_2,\be)^{-1}\zeta(\ga_1+\ga_2,\ga)^{-1}\zeta(\be_1,\be_2)
\zeta(\ga_1,\ga_2)^{-1}.
\end{align*}
Hence, $\lan y_1\circ y_2,x\ran_{q,\zeta}=\lan
y_1,x_1\ran_{q,\zeta}\lan y_2,x_2\ran_{q,\zeta}$.
\end{proof}

\begin{coro} Let $\lan\, ,\, \ran_{r,s}$ be the skew Hopf pairing on
$U_{r,s}(\mg)$ constructed in \cite{HP1} $($see also \cite{BW1, BGH1}$)$, we have
$\lan\, ,\, \ran_{q,\zeta}=\lan\, ,\, \ran_{r,s}$.
\end{coro}

\section{\bf Deformed representation theory}

In the above sections, we considered $U_{r,s}(\mg)$ as a deformation
structure of $U_{q,q^{-1}}(\mg)$. Correspondingly, we continue to
consider $U_{r,s}(\mg)$-modules as the deformation structure of
$U_{q,q^{-1}}(\mg)$-modules. To this end, we need to extend the
definitions of (skew)bicharacter $\zeta$ or $p_{ij}$ to be defined
on $\La\times \La$ such that
$\zeta(\lambda,\al_i)=\Pi_{j=1}^n\zeta(\al_j,\al_i)^{\frac{a_j}{m}}=\zeta(\al_i,\lambda)^{-1}$,
or $p_{\lambda,\al_i}=\Pi_{j=1}^n(p_{ji}^{\frac{1}m})^{a_j}$ for
$\lambda=\frac1{m}\sum_ja_j\al_j\in\La$.

\begin{defi}
The category $\O^{r,s}$ consists of finite-dimensional
$U_{r,s}(\mg)$-modules $V^{r,s}$ $($of type $1$$)$ satisfying the
following conditions:

$(1)$\ $V^{r,s}$ has a weight space decomposition
$V^{r,s}=\bigoplus_{\la\in\La}V^{r,s}_{\la}$, where
$$
V^{r,s}_{\la}=\{v\in V^{r,s}\mid
\om_{i}v=r^{\lan\la,\al_i\ran}s^{-\lan\al_i,\la\ran}v,\
\om_{i}'v=r^{-\lan\al_i,\la\ran}s^{\lan\la,\al_i\ran}v, \ \forall\;
i\in I\}
$$
and $\dim V_{\la}^{r,s}<\infty$ for all $\la\in\La$.

$(2)$\ there exist a finite number of weights $\la_1,\dots, \la_t\in
\La$ such that
$$
\wt (V^{r,s})\subset D(\la_1)\cup\cdots\cup D(\la_t),
$$
where $D(\la_i):=\{\mu\in\La\mid \mu<\la_i\}$. The morphisms are
taken to be usual $U_{r,s}(\mg)$-module homomorphisms.
\end{defi}

Let $r=q$ and $s=q^{-1}$, we get the category
$\O^{q}=\O^{q,q^{-1}}$.

\begin{pro}
Let $V^{q}\in\text{Ob}({\O^{q}})$, Then $V^{q}$ has a natural
$U_{q,\zeta}(\mg)$-module structure.
$$
x\cdot_{\zeta}v=\zeta(\al-\be,\la)\zeta(\al,\be)x.v,\quad \forall\;
x\in (U_{q})_{\al,\be},\ \forall\; v\in V_{\la}^{q}.
$$
Denote this module by $V^{q,\zeta}$.
\end{pro}
\begin{proof}
By the definition above,
\begin{align*}
K_i\cdot_{\zeta}v&=\zeta(2\al_i,\la)K_i.v=p_{\al_i\la}q^{(\al_i,\la)}v=r^{\lan\la,\al_i\ran}s^{-\lan\al_i,\la\ran}v,\\
K_i'\cdot_{\zeta}v&=\zeta(2\al_i,\la)K_i'.v=p_{\al_i\la}q^{-(\al_i,\la)}v=p_{\la\al_i}^{-1}q^{-(\al_i,\la)}v=r^{-\lan\al_i,\la\ran}s^{\lan\la,\al_i\ran}v,\\
E_i\cdot_{\zeta}v&=\zeta(\al_i,\la)E_i.v,\\
F_i\cdot_{\zeta}v&=\zeta(\al_i,\la)F_i.v.
\end{align*}
For any $v\in V^{q}_{\la}$ and $i\in I$,
\begin{align*}
K_i\cdot_{\zeta}(E_j\cdot_{\zeta}(K_i^{-1}\cdot_{\zeta}v))=r^{\lan
j, i\ran}s^{-\lan i, j\ran} \zeta(\al_j,\la)E_j.v=(K_i\circ E_j\circ
K_i^{-1})\cdot_{\zeta}v.
\end{align*}
Similarly, we can check the other relations.
\end{proof}

Let $\O^{q,\zeta}$ be the category consisting of
$U_{q,\zeta}(\mg)$-modules $V^{q,\zeta}$.
\begin{theo}
As braided tensor categories, $\O^{q}$  is equivalent to
$\O^{q,\zeta}$.
\end{theo}
\begin{proof}
Let $F^{\zeta}: V^{q}\lra V^{q,\zeta}$ be the functor between the
categories $\O^{q}$ and $\O^{q,\zeta}$. For any
$V^q,V'^q\in\text{Ob}({\O^{q}})$, there exists an isomorphism of
$U_{q,\zeta}(\mg)$-modules
$$
\xi_{V^q,V'^q}: (V^q\ot V'^q)^{\zeta}\lra V^{q,\zeta}\ot V'^{q,\zeta},
$$
where
$$
\xi_{V^q,V'^q}(v\ot v')=\zeta(\mu',\mu)v\ot v',\quad\forall\
v\in V_{\mu}^{q},\ v'\in V_{\mu'}'^{q}.
$$
Assume that $v\ot v'\in (V^q\ot V'^q)^{\zeta}_{\la}$, where
$\mu+\mu'=\la$. Then
\begin{align*}
\xi_{V^q,V'^q}(x\cdot_{\zeta}(v\ot v'))
&=\zeta(\al-\be,\la)\zeta(\al,\be)\xi_{V^q,V'^q}(x.(v\ot v'))\\
&= \zeta(\al-\be,\la)\zeta(\al,\be)\xi_{V^q,V'^q}(\De(x)(v\ot v'))
\end{align*}
\begin{align*}
&\quad= \zeta(\al-\be,\la)\zeta(\al,\be)\xi_{V^q,V'^q}\Bigl(\sum_{\ga} x_{\al,\ga}.v\ot x_{-\ga,\be}.v'\Bigr)\\
&\quad= \zeta(\al-\be,\mu+\mu')\zeta(\al,\be)\Bigl(\sum_{\ga}\zeta(\mu'+\be-\ga,\mu+\al+\ga) x_{\al,\ga}.v\ot x_{-\ga,\be}.v'\Bigr)\\
&\quad=\zeta(\mu',\mu)\zeta(\al,\mu)\zeta(\mu',\be)\Bigl(\sum_{\ga}\zeta(\mu+\mu'+\al+\be,\ga)
x_{\al,\ga}.v\ot x_{-\ga,\be}.v'\Bigr).
\end{align*}
On the other hand,
\begin{align*}
&x\cdot_{\zeta}(\xi_{V^q,V'^q}(v\ot v'))=\zeta(\mu,\mu')x\cdot_{\zeta}(v\ot v')\\
&\quad=\zeta(\mu',\mu)\Bigl(\sum x_{\al,\ga}\cdot_{\zeta}v\ot x_{-\ga,\be}\cdot_{\zeta}v'\Bigr)\\
&\quad=
\zeta(\mu',\mu)\Bigl(\sum_{\ga}\zeta(\al-\ga,\mu)\zeta(\al,\ga)\zeta(-\ga-\be,\mu')
\zeta(-\ga,\be)x_{\al,\ga}.v\ot x_{-\ga,\be}.v'\Bigr)\\
&\quad=
\zeta(\mu',\mu)\zeta(\al,\mu)\zeta(\mu',\be)\Bigl(\sum_{\ga}\zeta(\mu+\mu'+\al+\be,\ga)
x_{\al,\ga}.v\ot x_{-\ga,\be}.v'\Bigr).
\end{align*}
Hence,
$$
\xi_{V^q,V'^q}(x\cdot_{\zeta}(v\ot v'))=x\cdot_{\zeta}(\xi_{V^q,V'^q}(v\ot v')).
$$
That is, $\xi_{V^q,V'^q}$ is a homomorphism of
$U_{q,\zeta}$-modules. It is straightforward to prove that
$\xi_{V^q,V'^q}$ is an isomorphism. Next, we shall show that the
functor $F^{\zeta}$ preserves the braiding of $\O^{q}$. For any
$V^q,\,V'^q\in\text{Ob}(\O^{q})$, we define
$$
R^{q,\zeta}_{V^{q,\zeta},V'^{q,\zeta}}:=\xi_{V'^q,V^q}\circ
R^{q}_{V^q,V'^q}\circ \xi_{V^q,V'^q}^{-1}.
$$
Then $R^{q,\zeta}_{V^{q,\zeta},V'^{q,\zeta}}: V^{q,\zeta}\ot
V'^{q,\zeta}\lra V'^{q,\zeta}\ot V^{q,\zeta}$ is an isomorphism in
$\O^{q,\zeta}$ and $R^{q,\zeta}_{V^{q,\zeta},V^{q,\zeta}}$ satisfies
$$
R_{12}^{q,\zeta}R_{23}^{q,\zeta}R_{12}^{q,\zeta}=R_{23}^{q,\zeta}R_{12}^{q,\zeta}R_{23}^{q,\zeta}.
$$

This completes the proof.
\end{proof}

\begin{coro}
As braided tensor categories, we have the following equivalence.
\begin{eqnarray}
\O^{q}\simeq\O^{r,s}.
\end{eqnarray}
\end{coro}

\bibliographystyle{amsalpha}

\begin{thebibliography}{9999}

\bibitem{AE} N. Andruskiewitsch, B. Enriquez, \textit{Examples of compact matrix
pseudogroups arising from the twisting operation}, Comm. Math.
Phys. {\bf149} (1992), 195--207.




\bibitem{AST}M. Artin, W. Schelter, and J. Tate, \textit{Quantum deformations of
$GL(n)$}, Comm. Pure Appl. Math. {\textbf44} (1991), 879--895.


\bibitem{BGH1} N. Bergeron, Y. Gao and N. Hu, \textit{Drinfel'd doubles and
Lusztig's symmetries of two-parameter quantum groups}, J. Algebra,
\textbf{301} (2006), 378--405.


\bibitem{BGH2} N. Bergeron, Y. Gao and N. Hu, \textit{Representations of
two-parameter quantum orthogonal groups and symplectic groups},
arXiv:math/0510124, AMS/IP: Studies Adv. Math. {\bf 39}, (2007),
1--21.

\bibitem{BW1} G.
Benkart and S. Witherspoon, \textit{Two-parameter quantum groups (of
type $A$) and Drinfel'd doubles}, Algebr. Represent. Theory,
\textbf{7} (2004), 261--286.

\bibitem{BW2} G. Benkart and S. Witherspoon, \textit{Representations of two-parameter quantum
groups  and Schur-Weyl duality}, Hopf Algebras, pp.
65--92, Lecture Notes in Pure and Appl. Math., \textbf{237}, Dekker,
New York, 2004.


\bibitem{DT}Y. Doi and M. Takeuchi, \textit{Multiplication alteration by
two-cocyles: the quantum-version}, Comm. Algebra {\bf 22}, (1994),
5715--5732.

\bibitem{H}I. Heckenberger, \textit{The Weyl groupoid of a Nichols algebra of diagonal type}, Invent. Math. {\bf 164}, (2006),
 175--188.

\bibitem{HLT}T. Hodges, T. Levasseur, M. Toro, \textit{Algebraic structure
of multi-parameter quantum groups}, Adv. Math. 126 (1997), 52--92.

\bibitem{HZ} J. Hu, Y. Zhang, \textit{Quantum double of ${\rm U}_q((\mathfrak{sl}_2)^{\leq 0})$}, J. Algebra, {\bf 317}, (2007), 87--110.

\bibitem{HP1} N. Hu, Y. Pei,  \textit{Notes on two-parameter quantum groups, (I)}, \textbf{51} (6) (2008), 1101--1110. arXiv.math.QA/0702298.

\bibitem{HRZ} N. Hu, Y. Pei, M. Rosso, \textit{Multi-parameter quantum groups and
quantum shuffles, (I)}, Contemp. Math. (to appear), arXiv:0811.0129.





\bibitem{Ma}S. Majid,  \textit{Foundations of Quantum Group Theory},
Cambridge U.P, Cambridge, 1995.

\bibitem{M} S. Montgomery, \textit{Hopf Algebras and Their Actions on Rings}, CBMS Conf.
Math. Publ., \textbf{82}, Amer. Math. Soc., Providence, 1993.

\bibitem{P} Y. Pei,  \textit{Multiparameter quantized
enveloping algebras and their realizations}, Ph.~D. thesis, East
China Normal University, Shanghai, China, 2007.


\bibitem{Re} N. Reshetikhin, \textit{Multiparameter quantum groups and
twisted quasitriangular Hopf algebras}, Lett. Math. Phys. {\textbf
20}, (1990), pp. 331--335.

\bibitem{Ro1}  M. Rosso,  \textit{Quantum groups and quantum shuffles},
Invent. Math., \textbf{133}  (1998),  399--416.

\bibitem{W}
S. Westreich, \textit{Hopf algebras of type $A_n$, twistings and the
FRT-construction}, Algebr. Represent. Theory, {\bf 11}, (2008),
63--82.

\end{thebibliography}

\end{document}